\documentclass[12pt]{article}
\usepackage{graphicx}
\usepackage{color}
\usepackage{rotating}
\usepackage{amsmath}
\usepackage{amsfonts}
\usepackage{amssymb}
\usepackage{graphics,graphicx}
\usepackage{latexsym,longtable}
\usepackage[latin1]{inputenc}
\usepackage[official]{eurosym}
\usepackage{amsthm}
\usepackage[super]{nth}
 
\usepackage{tikz}
\usetikzlibrary{calc}
\usetikzlibrary{decorations.pathreplacing}
\usetikzlibrary{matrix,shapes.geometric,backgrounds}
\usepackage{pgfplots}
 
\newcommand{\E}{\mathrm{I\!E}}

\newtheorem{thm}{Theorem}[section]

\newtheorem{lem}[thm]{Lemma}
\newtheorem{rem}[thm]{Remark}
\newtheorem{ass}[thm]{Assumption}

\def\square{\ifmmode\sqr\else{$\sqr$}\fi}
\def\sqr{\vcenter{
\hrule height.1mm
\hbox{\vrule width.1mm height2.2mm\kern2.18mm
\vrule width.1mm}
\hrule height.1mm}}

\newtheorem{proposition}{Proposition}

\newtheorem{lemma}{Lemma}

\newtheorem{corollary}{Corollary}

\newcommand{\R}{\mathbb R}
\newcommand{\N}{\mathbb N}
\newcommand{\paren}[1]{\left(#1\right)}
\newcommand{\prect}[1]{\left[#1\right]}
\newcommand{\chav}[1]{\left\{#1\right\}}
\newcommand{\barra}[1]{\left. #1\right |}

\newcommand{\deriv}[2]{\frac{\partial #1}{\partial #2}}
\newcommand{\derivS}[2]{\frac{\partial ^2 #1}{\partial #2^2}}
\newcommand{\Indicatriz}[1]{\chi _{\chav{#1}}}
 
\newcommand{\h}{h}
 
\begin{document}
 
\title{Analytical solution to an investment problem under uncertainties with shocks
}

\author{\textsc{Cl{\'a}udia Nunes$^{1}$, Rita Pimentel$^{1}$} \\
$^{1}$\textit{CEMAT, T{\'e}cnico Lisboa, Universidade de Lisboa}\\
\textit{Av. Rovisco Pais, 1, 1049-001 Lisboa, Portugal}}
\maketitle
 
\begin{abstract}
We derive the optimal investment decision in a project where both demand and investment costs are stochastic processes, eventually subject to shocks. We extend the approach used in \cite{Dixit:Pindyck:94}, chapter 6.5, to deal with two sources of uncertainty, but assuming that the underlying processes are no longer geometric Brownian diffusions but rather jump diffusion processes. For the class of isoelastic functions that we address in this paper, it is still possible to derive a closed expression for the value of the firm. We prove formally that the result we get is indeed the solution of the optimization problem.
Finally, we derived comparative statistics for the investment threshold with respect to the relevant parameters.
\end{abstract}
%
%

\section{Introduction}
\label{sec:introduction}
 
A real option is the right, but not the obligation, to undertake certain business initiatives, such as deferring, abandoning, expanding, staging or contracting a capital investment project. 
Real options have three characteristics: the decision may be postponed; there is uncertainty about future rewards;
the investment cost is irreversible. There are several types of real options and different kinds of questions, but in all of them there is the issue regarding the optimal timing to undertake the decision (for instance, investment, abandon or switching). Consult \cite{Dixit:Pindyck:94} for a survey on real options analysis.

In this paper we propose a two-fold contribution to the state of the art of real options analysis. On one hand we assume that the uncertainty involving the decision depends on two stochastic processes: the demand  and the investment cost. On the other hand we allow the dynamics of these processes to be a combination of a diffusion part (with continuous sample path) with a jump part (introducing discontinuities in the sample paths). So, the model presented here expands the usual setting of options and adds new features, including the ability to model jumps in the demand process, and including stochastic investment costs.
 
In the early works in finance, regarding both finance options and investment decisions, the underlying stochastic process modeling the uncertainty was frequently assumed to be a continuous process, as a geometric Brownian motion (GBM, for short). There are many reasons for this, but this is largely due to the fact that the normal distribution, as well as the continuous-time process it generates have nice analytic properties \cite{Eberlein:10}.
 
Observations from real situations show that this assumption does not hold. The bar chart in Figure \ref{fig:apple} shows the global sales numbers of iPhones from the $\nth{3}$ quarter of 2007 to the end of 2014, with data aggregated by quarters (source: Statista 2014).
\begin{figure}
\centering
\includegraphics[scale=0.8]{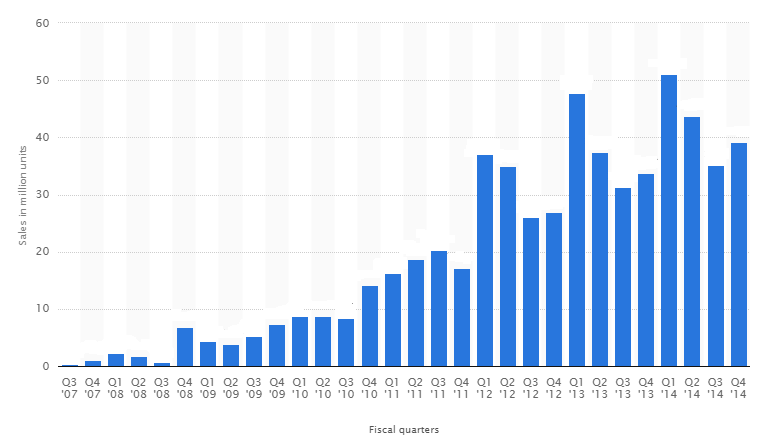}
\caption{Global Apple iPhone sales from $\nth{3}$ quarter 2007 to $\nth{4}$ quarter 2014 (in million units)}
\label{fig:apple}
\end{figure}
It is clear that in this case upward jumps occur with a certain frequency, that may be related with the launching of new versions of the iPhone.
 
Nowadays, with the world global market, there are exogenous events that may lead to a sudden increase or decrease in the demand for a certain product. Clearly, these shocks will considerably affect the decisions of the investors. Also the investment costs depend highly on other factors, such as the oil, raw materials  or labour prices. Therefore, when we assume the investment to be constant and known, we are simplifying the model. This problem is even more important when one considers large investments that take several years to be completed (e.g. construction of dams, nuclear power plants, or high speed railways).
 
The example regarding the number of sales of iPhones shows the existence of jumps in the demand process. In fact, empirical analysis shows there is a need to contemplate models that can take into account discontinuous behavior for a quite wide range of situations. For example, regarding investment decisions, in particular in the case of high-technology companies - where the fast advances in technology continue to challenge managerial decision-makers in terms of the timing of investments \cite{Wu:12} - there is a strong need to consider models that take into account sudden decreases or increases in the demand level \cite{Hagspiel:14}.

Morever, examples regarding jumps in the investment are also easy to find.  The jumps may represent uncertainties about the arrival and impact (on the underlying investment) of new information concerning technological innovation, competition, political risk, regulatory effects and other sources \cite{Martzoukos:02,Wu:07}. It is particularly important to consider jumps process in investment problems related with R\&D investments, as \cite{Koh:07} suggest. Furthermore, when one is deciding about investments in energy, which are long-term investments, one should accommodate the impact of climate change policies effect on carbon prices, as \cite{Yang:07} points out in their report. For instance,  the Horizon 2020 Energy Challenge is designed to support the transition to a secure, clean and efficient energy system for Europe. So, if a company decides to undertake any kind of such investment during this time period, probably will apply to special funds and thus the investment cost will likely be smaller. After 2020, there are not yet rules about the European funding of such projects. Therefore, not only the investment cost is random, but also will suffer a jump due to this new regulatory regime (see http://ec.europa.eu/easme/en/energy for more details).

Following this reasoning, in this paper we assume that the decision regarding the investment depends on two stochastic factors: one related with the demand process, that we will denote by ${\bf X}=\{X_t, t \geq 0\}$, and another one related with the investment cost, ${\bf I}=\{I_t, t \geq 0\}$. We allow that both processes may have discontinuous paths, due to the possibility of jumps driven by external Poisson processes. Later we will introduce in more detail the mathematical model and assumptions.

We also claim that we extend the approach provided in \cite{Dixit:Pindyck:94}, in Chapter 6.5. In this book, the authors consider that there are two sources of uncertainty (the revenue flow and the investment cost), both following a GBM, with (possible) correlated Brownian motions. Then the authors propose a change of variable that will turn this problem with two variables into a problem with one variable, analytically solvable. Here we follow the same approach, but with two different features: first the dynamics of the processes involved are no longer a GBM and, secondly, we prove analytically that the solution we get is in fact exactly the solution of the original problem with two dimensions. In order to prove this result we derive explicitly the quasi-variational inequalities emerging from the principle of dynamic programming, obtaining the corresponding Hamilton-Jacobi-Bellman (HJB) equation. Moreover, using a verification theorem, we prove that the proposed solution is indeed the solution of the original problem \cite{yong:99,Oksendal:03}.

The remainder of the paper is organized as follows. In Section 2 we present the framework rationale and the valuation model, and in Section 3 we formalize the problem as an optimal stopping one, we present our approach and provide the optimal solution. In Section 4 we derive results regarding the behavior of the investment threshold as a function of the most relevant parameters and, finally, Section 5 concludes the work, presenting also some recommendations for future extensions.
 
\section{The model}
 
In this section we introduce the mathematical model along with the assumptions that we use in order to derive the investment policy. For that, we start by presenting the dynamics of the stochastic processes involved, namely, the demand and the investment expenditures processes. We also introduce the valuation model we use to derive the optimal investment problem that we address in this paper.
 
\subsection{The processes dynamics}
 
We assume that there are two sources of uncertainty, that we will denote by $\boldsymbol{X}$ and $\boldsymbol{I}$, representing the demand and the investment, respectively. Moreover, we consider that both have continuous sample paths almost everywhere but occasionally jumps may occur, causing a discontinuity in the path.
 
%
 
 
Generally speaking, let $\boldsymbol{Y} = \chav{Y_t:t\geq 0}$ denote a jump-diffusion process, which has two types of changes: the "normal" vibration, which is represented by a GBM, with continuous sample paths, and the "abnormal", modeled by a jump process, which introduce discontinuities.
 
The jumps occur in random times and are driven by an homogeneous Poisson process, hereby denoted by $\chav{N_t, t \geq 0}$, with rate $\lambda$. Moreover, the percentages of the jumps are themselves random variables, that we denote by $\chav{U_i}_{i \in \N} \overset{i.i.d.}{\thicksim} U$, i.e., although they may be of different magnitude, they all obey the same probability law and are independent from each others. We let $U_i >-1$, allowing for positive and for negative jumps.
 
Thus we may express $Y_t$ as follows:
\begin{equation}\label{processY}
Y_t = y_0 \exp {\prect{\paren{\mu-\frac{\sigma^2}{2}-\lambda m}t+\sigma W_t}} \prod_{i=1}^{N_t}\paren{1+U_i},
\end{equation}
where $y_0$ is the initial value of the process, i.e., $Y_0=y_0$, $\mu$ and $\sigma > 0$ represent, respectively, the drift and the volatility of the continuous part of the process and $\chav{W_t, t \geq 0}$ is a Brownian motion process. Furthermore, the processes $\chav{W_t, t \geq 0}$ and $\chav{N_t, t \geq 0}$ are independent and are also both independent of the sequence of random variables $\chav{U_i}_{i \in \N}$. We use the convention that if $N_t = 0$ for some $t \geq 0$, then the product in equation \eqref{processY} is equal to 1. Note that considering the parameters associated with the jumps equal to zero, we obtain the usual GBM.
 
As both the Brownian motion and the Poisson processes are Markovian, it follows that $\boldsymbol{Y}$ is also Markovian. It is also stationary and its $k^{\text{th}}$ order moments are provided in the next lemma (see Appendix (\ref{moments}) for the proof).
 
\bigskip
 
\begin{lemma}\label{propEY}
For $\boldsymbol{Y} = \chav{Y_t:t\geq 0}$, with $Y_t$ defined as in \eqref{processY}, $t, s \geq 0$ and $k \in \N$,
\begin{equation*}
\E \prect{\barra{ Y_{t+s}^k} {Y_t}} = Y_t^k \; \exp \chav{ {\prect{\paren{\mu+(k-1)\frac{\sigma^2}{2}-\lambda m} k + \lambda \paren{\E \prect{\paren{1+U}^k}-1} }s}},
\end{equation*}
where $\E \prect{\barra{ .} {Y_t}}$ denotes the conditional expectation given the value of the process $\boldsymbol{Y}$ at time $t$.
\end{lemma}
 
Returning to the investment problem: in the sequel we assume that both processes $\boldsymbol{X}$ and $\boldsymbol{I}$ follow jump-diffusion processes, such that
%
\begin{equation*}
X_t = x_0 \exp {\prect{\paren{\mu_X-\frac{\sigma_X^2}{2}-\lambda_X m_X}t+\sigma_X W^X_t}} \prod_{i=1}^{N^X_t}\paren{1+U^X_i}
\end{equation*}
and
\begin{equation*}
I_t = i_0 \exp {\prect{\paren{\mu_I-\frac{\sigma_I^2}{2}-\lambda_I m_I}t+\sigma_I W^I_t}} \prod_{i=1}^{N^I_t}\paren{1+U^I_i}.
\end{equation*}
We use an index $X$ or $I$ for each parameter to distinguish between the two processes. Finally, we assume that the both processes are independent from each other.

Next we introduce the valuation model, that defines the production function for the problem that we propose to address.

\subsection{The valuation model}
 
In this paper we propose to solve the question of when to invest in a new project, when both the demand and the investment are stochastic and subject to shocks. Furthermore, we assume that the investment may take some time to really starts to pay back, and therefore there may be a gap of $n$ time units between the decision and the effectiveness of the investment.
 
We also assume that the profit of the firm depends on the stochastic process of the demand, and we use the following notation: $V_0$ ($V_1$) is the profit of the firm before (after) the investment. We consider an isoelastic function:
\begin{equation*}
V_{j}(X_{t}) = \kappa_{j}X_{t}^{\theta},
\end{equation*}
for $j \in \chav{0,1}$, with $\kappa_1 > \kappa_0 > 0$ and $\theta > 0$. So the investment in the new project only changes the model coefficients, $\kappa_0$ and $\kappa_1$, and not the elasticity parameter, $\theta$. Furthermore,
\begin{equation}\label{Dfunction}
D(X_t)=\paren{\kappa_1-\kappa_0} X_{t}^{\theta}
\end{equation}
is the profit's difference before and after the investment.
 
Denoting the risk-neutral discount rate by $\rho$, then the value of the firm if the investment is decided at time $\tau$ is given by:
\begin{equation*}
\int_{0}^{\tau+n} V_{0}(X_{t})e^{-\rho t} dt +\int_{\tau+n}^{+\infty }V_{1}(X_{t})e^{-\rho t}dt - I_{\tau} e^{-\rho \tau},
\end{equation*}
as long as $\tau< + \infty$.
 
The goal of the firm is to maximize its expected value, with respect to the time to decide to invest, that is, the firm seeks to maximize the following expected value
\begin{align*}
&\E\prect{ \left. \int_{0}^{\tau+n} V_{0}(X_{t})e^{-\rho t} dt + \chav{\int_{\tau+n}^{+\infty }V_{1}(X_{t})e^{-\rho t}dt - I_{\tau} e^{-\rho \tau}} \Indicatriz{\tau<+\infty} \right | (X_0=x_0,I_0=i_0)}\\
&= \E^{(x_0,i_0)} \prect{\int_{0}^{\tau+n} V_{0}(X_{t})e^{-\rho t} dt + \chav{\int_{\tau+n}^{+\infty }V_{1}(X_{t})e^{-\rho t}dt - I_{\tau} e^{-\rho \tau}} \Indicatriz{\tau<+\infty}},
\end{align*}
where in the rest of the paper $\E^{(x_0,i_0)}$ denotes the expectation conditional on $(X_0,I_0)=(x_0,i_0)$; moreover $\Indicatriz{A}$ represents the indicator function of the set $A$ and $\tau$ is a stopping time for the filtration generated by the bidimensional process $({\bf X},{\bf I})$.
 
Using \eqref{Dfunction} and simple manipulations, we can re-write the previous expectation as
\begin{equation*}
\E^{(x_0,i_0)} \prect{\int_{0}^{+\infty}V_{0}(X_{t})e^{-\rho t}dt} + \E^{(x_0,i_0)} \prect{\chav{\int_{\tau+n}^{+\infty } D(X_t)e^{-\rho t}dt - I_{\tau} e^{-\rho \tau}} \Indicatriz{\tau<+\infty}}.
\end{equation*}
As the first part of the expression does not depend on the time $\tau$ to invest, then the goal is to maximize the following functional:
\begin{equation*}
J^{\tau}(x_0,i_0) = \E^{(x_0,i_0)} \prect{\chav{\int_{\tau+n}^{+\infty } D(X_t)e^{-\rho t}dt - I_{\tau} e^{-\rho \tau}} \Indicatriz{\tau<+\infty}},
\end{equation*}
which represents the value of the option when the firm exercises it at time $\tau$, with $\tau< + \infty$, given that the initial state is $(x_0,i_0)$. Following \cite{Oksendal:03}, we call $J^{\tau}$ the {\em performance criterion}.
 
Note that in order to have $\tau< + \infty$, one needs to assume that $ \E^{(x_0,i_0)} \prect{\int_{0}^{+\infty} V_{0}(X_{t})e^{-\rho t}dt} < + \infty$, which happens if and only if
\begin{equation*}
\label{h:def}
\h  >0,
\end{equation*}
where
\begin{equation}
\label{def:h}
h= \rho - \paren{\mu_X+(\theta-1) \frac{\sigma_X^2}{2}-\lambda_X m_X} \theta - \lambda_X \paren{\E \prect{\paren{1+U_X}^{\theta}}-1}
\end{equation}
since
\begin{equation*}
\E^{(x_0,i_0)} \prect{\int_{0}^{+\infty} V_{0}(X_{t})e^{-\rho t}dt} = \kappa_{0} x_0^{\theta} \int_{0}^{+\infty} e^{-\h t} dt.
\end{equation*}
From now on, we will be assuming such restriction in the parameters, to avoid trivialities.
%
%
Next we state and prove a result related with the strong Markov property of the involved processes, that we will use subsequently in the optimization problem.
 
\begin{proposition}
The performance criterion can be re-written as follows:
\begin{equation*}
J^{\tau}(x,i) = \mathbb{E}^{(x,i)} \prect{e^{-\rho \tau} {g(X_{\tau},I_{\tau})} \chi_{\chav{\tau < +\infty}}},
\end{equation*}
with
\begin{equation*}\label{g}
{g(x,i)}=\paren{\kappa_1 -\kappa_0} A x^{\theta} - i
\end{equation*}
and
\begin{equation}\label{A}
A = \frac{ e^{-\h n} }{\h}.
\end{equation}
\end{proposition}
\begin{proof}
Applying a change of variable in the performance criterion, we obtain
\begin{equation*}
J^{\tau}(x,i) = \E^{(x,i)} \prect{ e^{-\rho \tau} \chav{\int_{0}^{+\infty } D(X_{\tau + t + n})e^{-\rho (t + n)}dt - I_{\tau}} \Indicatriz{\tau<+\infty}} .
\end{equation*}
Then, using conditional expectation and the strong Markov property, we get
\begin{equation*}
J^{\tau}(x,i) = \mathbb{E}^{(x,i)} \prect{e^{-\rho \tau} \chav{\E \prect{\barra{\int_{0}^{+\infty } D(X_{\tau + t + n})e^{-\rho (t + n)}dt} {(X_\tau,I_\tau)}} - I_{\tau} } \Indicatriz{\tau<+\infty}}.
\end{equation*}
By Fubini's theorem and the independence between $\boldsymbol{X}$ and $\boldsymbol{I}$, it follows that:
\begin{equation}\label{EintD}
\E \prect{\barra{\int_{0}^{+\infty } D(X_{\tau + t + n})e^{-\rho (t + n)}dt } {(X_\tau,I_\tau)}} = \int_{0}^{+\infty } \E \prect{\barra{ D(X_{\tau + t + n})} {X_\tau}} e^{-\rho (t + n)} dt.
\end{equation}
From Lemma \ref{propEY} and from simple calculations, we have:
\begin{equation}\label{ED}
\E \prect{\barra{ D(X_{\tau + t + n})} {X_\tau}} = \paren{\kappa_1-\kappa_0} X_\tau^{\theta} \; e^{(\rho - \h)(t+n)}.
\end{equation}
Plugging \eqref{ED} on \eqref{EintD} results in:
\begin{equation*}
\E \prect{\barra{\int_{0}^{+\infty } D(X_{\tau + t + n})e^{-\rho (t + n)}dt } {(X_\tau,I_\tau)}} = \paren{\kappa_1-\kappa_0} X_\tau^{\theta} \frac{ e^{-\h n} }{\h}.
\end{equation*}
Therefore, in view of the definition of $J^{\tau}$ and $A$, we get the pretended result.
\end{proof}
%

\section{The optimal stopping time problem}
 
The goal of the firm is to find the optimal time to invest in this new project, i.e., for every $x>0$ and $i>0$, we want to find $V(x,i)$ and $\tau^{\star} \in \mathcal{T}$ such that:
\begin{equation*}
V(x,i) = \sup_{\tau \in \mathcal{T}} J^{\tau}(x,i) = J^{\tau^{\star}}(x,i), \quad (x,i) \in \R^+ \times \R^+
\end{equation*}
where $\mathcal{T}$ is the set of all stopping times adapted to the filtration generated by the bidimensional stochastic process $\paren{\boldsymbol{X} , \boldsymbol{I}}$.
The function $V$ is called the \textit{value function}, which represents the optimal value of the firm, and the stopping time $\tau^\star$ is called an \textit{optimal stopping time}.
 
Using standard calculations from optimal stopping theory (see \cite{Oksendal:03}), we derive the following HJB  equation for this problem:
\begin{equation}\label{HJBoriginal}
min \chav{\rho V(x,i) - \mathcal{L} V(x,i) ,V(x,i)-g(x,i)}=0, \quad \forall \: (x,i) \in \R^+ \times \R^+ ,
\end{equation}
where $\mathcal{L}$ denotes the infinitesimal generator of the process $\paren{\boldsymbol{X},\boldsymbol{I}}$, which in view of \cite{Oksendal:03}, it is equal to
\begin{eqnarray*}
\mathcal{L} V (x,i) &=& \frac{\sigma_X^2}{2} x^2 \derivS{V(x,i)}{x} + \frac{\sigma_I^2}{2} i^2 \derivS{V(x,i)}{i}+ \paren{\mu_X-\lambda_X m_X} x \deriv{V(x,i) }{x}+ \paren{\mu_I-\lambda_I m_I} i \deriv{V(x,i) }{i}\\
&& +\lambda_X \prect{\E_{U^X}\paren{V(x(1+U^X),i)} - V(x,i)} +\lambda_I \prect{\E_{U^I}\paren{V(x,i(1+U^I))} - V(x,i)}.
\end{eqnarray*}
Note that in this equation we used the notation $\E_{U^X}$ and $\E_{U^I}$ to emphasize the meaning of such expected values; for instance $\E_{U^I}$ means that we are computing the expected value with respect to the random variable $U^I$. Whenever the meaning is clear from the context, we will simplify the notation and will use just $\E$ instead.

The HJB equation represents the two possible decisions that the firm has.
The first term of the equation corresponds to the decision to continue with the actual project and therefore postpone the investment decision. The second term corresponds to the decision to invest and hence the end of the decision problem. Thus
\begin{equation*}
\mathcal{C} = \chav{(x,i) \in \R^+ \times \R^+ : V(x,i)>g(x,i)}
\end{equation*}
is the set of values of demand and investment cost such it is optimal to continue with the actual project and, for that reason, is usually known as {\em continuation region}, whereas in the {\em stopping region}, $\mathcal{S}$, the value of the firm is exactly the gain if one decides to exercise the option.
The sets $\mathcal{C}$ and $\mathcal{S}$ are complementary on $\R^+ \times \R^+$. We define
\begin{equation*}
\tau^\star= \inf \chav{t \geq 0 : (X_t,I_t) \notin \mathcal{C}},
\end{equation*}
which is the first time the state variables corresponding to the demand and the investment cost are outside the continuation region.
Note that \eqref{HJBoriginal} means that
\begin{equation*}
\rho V(x,i) - \mathcal{L} V(x,i) \geq 0 \; \wedge \;
V(x,i) \geq g(x,i), \; \forall \; (x,i) \in \R^+ \times \R^+ .
\end{equation*}
%
Moreover,
\begin{equation*}
V(x,i) \geq g(x,i) \; \wedge \; \rho V(x,i) - \mathcal{L} V(x,i) = 0, \; \forall \; (x,i) \in \mathcal{C}
\end{equation*}
whereas
\begin{equation*}
V(x,i)=g(x,i) \; \wedge \; \rho g(x,i) - \mathcal{L} g(x,i) \geq 0, \; \forall \; (x,i) \in \mathcal{S} .
\end{equation*}
The solution of the HJB equation, $V$, must satisfy the following initial condition,
\begin{equation}\label{initialCondition}
V(0^+,i)=0, \quad \forall i \in \R^{+},
\end{equation}
which reflects the fact that the value of the firm will be zero if the demand is also zero.
Furthermore, the following value-matching and smooth-pasting conditions should also hold
\begin{equation}\label{boundaryConditions}
V(x,i)=g(x,i) \quad \text{and} \quad
\nabla V(x,i)=\nabla g(x,i)
\end{equation}
for $(x,i) \in \partial \mathcal{C}$, where $\partial \mathcal{C}$ denotes the boundary of $C$, which we call {\em critical boundary}, and $\nabla$ is the gradient function. In this particular case, with two sources of uncertainty, the critical boundary is a threshold curve separating the two regions.
 
So in order to solve the investment problem, we need to find the solution to the following partial differential equation:
\begin{equation}
\label{pde}
\rho V(x,i) - \mathcal{L} V(x,i) = 0
\end{equation}
in the continuation region and, simultaneously, to identify the continuation $\mathcal{C}$ and stopping $\mathcal{S}$ regions.
 
Both questions are challenging: on one hand the solution to \eqref{pde} cannot be found explicitly, and on the other hand we need to guess the form of the continuation region and prove that the solution proposed for the HJB equation satisfies the verification theorem.
Hereafter, we propose an alternative way to solve this problem, that circumvent the difficulties posed by the fact that we have two sources of uncertainty.
 
\subsection{Splitting the state space}
In this section we start by guessing the shape of the continuation region and latter we will prove indeed that our guess is the correct one, using the verification theorem.
 
Our guess about the continuation region comes from intuitive arguments: high demand and low investment cost should encourage the firm to invest, whereas in case of low demand and high investment cost it should be optimal to continue, i.e., postpone the investment decision. See Figure \ref{fig:InitialSpace} for the general plot of both regions.
\begin{center}
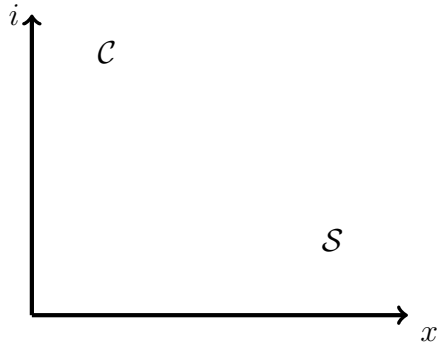
\begin{figure}[h!]
\centering
\begin{tikzpicture}
\draw[thick,->,ultra thick] (0,0) -- (5,0) node[anchor=north west] {$x$};
\draw[thick,->,ultra thick] (0,0) -- (0,4) node[left] {$i$};
\node[align=center] at (4,1) {$\mathcal{S}$};
\node[align=center] at (1,3.5) {$\mathcal{C}$};
\end{tikzpicture}
\caption{Space of process $(\boldsymbol{X},\boldsymbol{I})$, $\mathcal{C}$ (continuation region) and $\mathcal{S}$ (stopping region).}
\label{fig:InitialSpace}
\end{figure}
\end{center}
But we need to define precisely the boundary of $\mathcal{C}$ and for that we use the conditions derived from the HJB equation.
We consider the set
\begin{equation*}
U = \chav{(x,i) \in \R^+ \times \R^+ : \rho g(x,i) - \mathcal{L} g(x,i) < 0 }
\end{equation*}
which, taking into account the expression of $\mathcal{L} $ (the infinitesimal generator) and the $g$ function, can be equivalently expressed (after some simple calculations) as
\begin{equation*}
U = \chav{(x,i) \in \R^+ \times \R^+ : \frac{x^{\theta}}{i} < \frac{\rho - \mu_I}{\paren{\kappa_1 -\kappa_0} A \h} }.
\end{equation*}
By Propositions 3.3 and 3.4 of \cite{Oksendal:03}, we know that $U \subseteq \mathcal{C}$. Further, if $U=\emptyset$ then $V(x,i)=g(x,i), \; \forall \; (x,i) \in \R^+ \times \R^+$ and $\tau^\star=0$. This case would mean that it is optimal to invest right away. So we need to check under which conditions this trivial situation does not hold.
We have $\frac{\rho - \mu_I}{\paren{\kappa_1 -\kappa_0} A \h} \leq 0$ if and only if $\rho - \mu_I \leq 0$, as the denominator is positive; as $\frac{x^{\theta}}{i} > 0$, this would imply that $U=\emptyset$ and $\tau^\star=0$. Therefore, if $ \mu_I \geq \rho$ the firm should invest right away, which is coherent with a financial interpretation: if the drift of the investment cost is higher than the discount rate, the rational decision is to invest as soon as possible.
To avoid such case, we assume that this does not hold, i.e.,
\begin{equation}\label{restriction2}
\rho > \mu_I.
\end{equation}
Next we guess that the boundaries of $U$ and $\mathcal{C}$ have the same shape,  and therefore $\mathcal{C}$ may be written as follows:
\begin{equation*}
\mathcal{C} = \chav{(x,i) \in \R^+ \times \R^+ : 0 < \frac{x^{\theta}}{i} < q^\star},
\end{equation*}
where $q^\star$ is a trigger value that satisfies
\begin{equation}\label{qstarRestricition}
q^\star \geq \frac{\rho - \mu_I}{\paren{\kappa_1 -\kappa_0} A \h}.
\end{equation}
In Figure \ref{fig:U} the thick line represents the boundary of the set $U$ and the other dashed lines are some possible boundaries of the continuation region, $\mathcal{C}$ (as $U \subseteq\mathcal{C}$, which would be the case).
%
\begin{center}
\begin{figure}[h!]
\centering
\begin{tikzpicture}
\draw[thick,->,ultra thick] (0,0) -- (5,0) node[anchor=north west] {$x$};
\draw[thick,->,ultra thick] (0,0) -- (0,4) node[left] {$i$};
\node[align=center] at (1,3.5) {$U$};
\draw[domain=0:2.4,smooth,variable=\x,ultra thick,dashed] plot ({\x},{(1/1.5)*\x*\x});
\draw[domain=0:3.3,smooth,variable=\x,dashed] plot ({\x},{(1/3)*\x*\x});
\draw[domain=0:4.2,smooth,variable=\x,dashed] plot ({\x},{(1/5)*\x*\x});
\end{tikzpicture}
\caption{Set $U$ (bold line) and some possible boundaries of the continuation region, $\mathcal{C}$.}
\label{fig:U}
\end{figure}
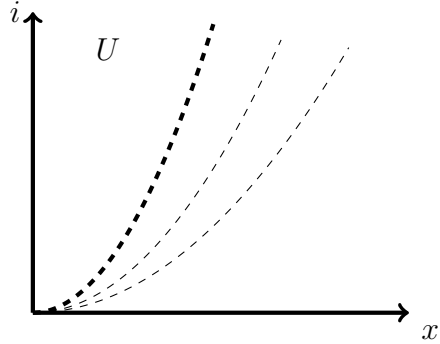
\end{center}

\subsection{Reducing to one source of uncertainty}
 
It follows from the last section that the decision between continuing and stopping depends on the demand level and investment cost only through a function of the two of them:
$$q=\frac{x^{\theta}}{i}.$$
Then the $g$ function present in the performance criteria can also be written in terms of this new variable, as $g(x,i) = i l(q)$, with
\begin{equation}\label{l}
l(q)=\paren{\kappa_1 -\kappa_0} A q -1
\end{equation}
and, in view of \eqref{l}, we propose that $$V(x,i)=i f\paren{q},$$
where $f$ is an appropriate function, that we need to derive still.
 
%
%
%
%
 
In chapter 6.5 of \cite{Dixit:Pindyck:94} is discussed a problem with two sources of uncertainty following a GBM. Using economical arguments, they propose a transformation of the two state variables into a one state variable, and thus reducing the problem to one dimension, that can be solved using a standard approach.
 
In the problem that we address we don't use any kind of economical arguments but we rely on the definition of the set $U$. Furthermore, the state variables that we consider in this paper follow a jump-diffusion process, and thus the dynamics is more general than the one proposed by \cite{Dixit:Pindyck:94}. Therefore, we may see our result as a generalization of the results provided by \cite{Dixit:Pindyck:94}.

 
%
%
 
 
Following the approach proposed in the previous section, we write the original HJB equation now for the case of just one variable:
\begin{equation}\label{HJB1uncert}
min \chav{\rho f(q) - \mathcal{L}_Q f(q) ,f(q) - l(q)}=0, \quad \forall \: q \in \R^+ ,
\end{equation}
where the infinitesimal generator of the unidimensional process $\boldsymbol{Q}=\chav{Q_t=\frac{X_t^{\theta }}{I_t}, t \geq 0}$ is
\begin{eqnarray}
\nonumber
\mathcal{L}_Q f(q) &=& \frac{1}{2} \paren{\theta^2\sigma_X^2 + \sigma_I^2} q^2 f^{\prime \prime}(q) + \prect{\paren{\mu_X + (\theta - 1) \frac{\sigma_X^2}{2} -\lambda_X m_X} \theta-\paren{\mu_I-\lambda_I m_I}} q f^\prime (q)\\
\nonumber
&& + \prect{\paren{\mu_I-\lambda_I m_I}-\paren{\lambda_X + \lambda_I }} f(q) \\
\label{LQ}
&& +\lambda_X \E\prect{ f\paren{q \paren{1+U^X}^{\theta}}} +\lambda_I \E\prect{(1+U^I) f \paren{q \paren{1+U^I}^{-1}}},
\end{eqnarray}
where this equation comes from ${\mathcal L}$ definition, the relationship between $V$ and $f$, and simple but tedious computations.
 
The corresponding continuation and stopping regions (hereby denoted by $\mathcal{C}_Q$ and $\mathcal{S}_Q$, respectively) can now be written as depending only on $q$:
\begin{equation*}
\mathcal{C}_Q = \chav{q \in \R^+: 0 < q < q^\star} \; \; \text{and} \; \; \mathcal{S}_Q = \chav{q \in \R^+: q \geq q^\star},
\end{equation*}
and thus $q^\star$ is the boundary between the continuation and the stopping regions, as it is usually the case in a problem with one state variable. We note that  now we are in a standard case problem of investment with just one state variable (see, for instance, chapter 5 of \cite{Dixit:Pindyck:94}) but with a different dynamics than the usual one (the process $\boldsymbol{Q}$ is not a GBM).

%
Moreover, the following conditions should hold
\begin{equation}\label{initial1uncert}
f(0^+)=0, \quad f(q^\star)=l(q^\star) \quad \text{and} \quad \barra{f^\prime(q)=l^\prime(q)}_{q=q^\star}.
\end{equation}
Next, using derivations familiar to the standard case, we are able to derive the analytical solution to equation \eqref{HJB1uncert}.
\begin{proposition}\label{Prop1Source}
The solution of the HJB Equation \eqref{HJB1uncert} verifying the conditions \eqref{initial1uncert}, hereby denoted by $f$, is given by
\begin{equation}
\label{ffunction}
f(q)=
\begin{cases}
\frac{1}{r_0-1} \paren{\frac{q}{q^\star}}^{r_0} & 0 < q <q^{\star} \\
\paren{\kappa_1 -\kappa_0} A q -1 & q \geq q^{\star}
\end{cases}
\end{equation}
where
\begin{equation}\label{qstar}
q^\star = \frac{r_0}{\paren{\kappa_1 -\kappa_0} A\paren{r_0-1}}
\end{equation}
and $r_0$ is the positive root of the function $j$:
\begin{eqnarray}
\label{j:def}
j(r)&=&\prect{\frac{\sigma_X^2}{2}\theta^2 + \frac{\sigma_I^2}{2}} r^2 + \prect{\paren{\mu_X -\frac{\sigma_X^2}{2}-\lambda_X m_X} \theta-\paren{\mu_I - \frac{\sigma_I^2}{2} -\lambda_I m_I}} r \nonumber \\
&& + \paren{\mu_I-\lambda_I m_I}-\paren{\lambda_X + \lambda_I }-\rho +\lambda_X \E\prect{(1+U^X)^{r \theta}}+\lambda_I \E\prect{(1+U^I)^{1-r}}. \label{j}
\end{eqnarray}
\end{proposition}
 
\begin{proof}
For the case $q \geq q^\star$ (corresponding to the stopping region), the value function is equal to $l$, the investment cost, defined on \eqref{l}, and thus the result is trivially proved. We then must prove the result for the continuation region, for which $q<q^\star$.
 
In this case, we propose that the solution of the HJB in the continuation region is of the form $\zeta(q)=b q^{r_0}$, where $b$ and $r_0$ are to be derived.
As $\zeta(0)=0$ (from the condition \eqref{initial1uncert}), it follows that $r_0$ needs to be positive.
 
Furthermore, in view of the definition of $\mathcal{L}_Q$ \eqref{LQ}, if we apply it to the proposed function $\zeta$, it follows that:
\begin{equation*}
\rho \zeta(q) - \mathcal{L}_Q \zeta(q) = - bq^{r_0} j(r_0),
\end{equation*}
where $j$ is defined on \eqref{j}. Therefore, in the continuation region
\begin{equation*}
\rho \zeta(q) - \mathcal{L}_Q \zeta(q) = 0 \Leftrightarrow j(r_0)=0,
\end{equation*}
meaning that $r_0$ is a positive root of $j$. Next we show that $j$ has one and only one positive root and thus $r_0$ is unique.
The second order derivative of $j$ is equal to
\begin{equation*}
j^{\prime \prime}(r)=\sigma_X^2 \theta ^2 + \sigma_I^2 +\lambda_X \theta ^2 \E\prect{ \prect{\ln(1+U^X)}^2 (1+U^X)^{r \theta }} +\lambda_I \E\prect{\prect{\ln(1+U^I)}^2 (1+U^I)^{1-r}}.
\end{equation*}
Given that the volatilities and the elasticity are positive, $U^X>-1$ a.s. and $U^I>-1$ a.s., it follows that $j^{\prime \prime} (r) >0, \; \forall r \in \R$. Thus, $j$ is a strictly convex function. Further, $lim_{r\rightarrow +\infty} j(r) = + \infty$ and $j(0) = \mu_I - \rho<0$ (by assumption \eqref{restriction2}), which means that $j(0) < 0$. Then, since $j$ is continuous, it has an unique positive root.
 
It remains to derive $b$. For that we use conditions presented in \eqref{initial1uncert}, and straightforward calculations lead to:
\begin{equation*}
b = \frac{1}{r_0-1}\paren{\frac{1}{q^\star}}^{r_0},
\end{equation*}
where $q^\star$ must be given by \eqref{qstar}.
In order to ensure that this leads to an admissible solution for $q^{\star}$ (as $q^\star$ must be positive), one must have $r_0>1$.
%
%
%

The remainder of the proof is to check that \eqref{ffunction} is the solution of the HJB Equation \eqref{HJB1uncert}. For that we need
\begin{itemize}
\item[i)] To prove that in the continuation region $f(q) \geq l(q)$.
In the continuation region $f(q) = \frac{1}{r_0 - 1} \paren{\frac{q}{q^\star}}^{r_0}$.
So, we define the function
\begin{equation*}
\varphi(q) = \frac{1}{r_0 - 1} \paren{\frac{1}{q^\star}}^{r_0} q^{r_0} - \paren{\kappa_1 -\kappa_0} A q + 1
\end{equation*}
with derivatives
\begin{equation*}
\varphi^\prime (q) = \frac{r_0}{r_0 - 1} \paren{\frac{1}{q^\star}}^{r_0} q^{r_0-1} - \paren{\kappa_1 -\kappa_0} A \quad \text{and} \quad \varphi^{\prime \prime}(q) = r_0 \paren{\frac{1}{q^\star}}^{r_0} q^{r_0-2}.
\end{equation*}
Since $\varphi^{\prime \prime}(q)>0$, $\varphi$ is a strictly convex function. As $\varphi^\prime (q) = 0 \Leftrightarrow q = q^\star$, then $q^\star$ is the unique minimum of the function $\varphi$ with $\varphi \paren{q^\star} = 0$.
 
Summarizing, the function $\varphi$ is continuous and strictly convex, has an unique minimum at $q^\star$ and $\varphi \paren{q^\star} = 0$, then $\varphi (q) \geq 0, \forall q \in \R^+$.
Therefore, $f(q) \geq l(q)$, for $\forall q \in \R$.
 
%
 
\item[ii)] To prove that in the stopping region $\rho f(q) - \mathcal{L}_Q f(q) \geq 0$.
 
In the stopping region $f(q)=l(q)$, then
\begin{equation*}
\rho f(q) - \mathcal{L}_Q f(q) = \paren{\kappa_1 -\kappa_0} A \h q + \paren{\mu_I - \rho}.
\end{equation*}
Given that $q \geq q^\star$ and using the inequality \eqref{qstarRestricition}, we conclude that
\begin{equation*}
\rho \zeta(q) - \mathcal{L}_Q \zeta(q) \geq \paren{\kappa_1 -\kappa_0} A \h q^\star + \paren{\mu_I - \rho} \geq \paren{\kappa_1 -\kappa_0} A \h \frac{\rho - \mu_I}{\paren{\kappa_1 -\kappa_0} A \h} + \mu_I - \rho=0.
\end{equation*}
\end{itemize}
 

Therefore, we conclude that the function $f$ is indeed the solution of the HJB Equation \eqref{HJB1uncert}.
\end{proof}

We have already seen that in order to avoid the trivial case of investment right away, one must have $\rho > \mu_I $. Moreover, from the previous proof of the proposition, one must also assume that $r_0>1$, which implies that $j(1) <0$, condition equivalent to $\h > \sigma^2_I $, where $h$ is defined in \eqref{def:h}.

\begin{ass}
\label{ass}
We assume that the following conditions on the parameters hold:
\begin{equation*}
\rho > \mu_I \; \; \text{and} \; \; \h > \sigma^2_I .
\end{equation*}
\end{ass}

\subsection{Original problem}
In this section we prove that the solution of the original problem can be obtained from the solution of the modified one.
 
\begin{corollary}
The solution of the HJB Equation \eqref{HJBoriginal} verifying the initial condition \eqref{initialCondition} and the boundary conditions \eqref{boundaryConditions}, hereby denoted by $V$, is given by
\begin{equation}\label{V}
V(x,i)=
\begin{cases}
\prect{\frac{\paren{\kappa_1 -\kappa_0} A x^{\theta}}{r_0}}^{r_0} \paren{\frac{r_0-1}{i}}^{r_0-1} &0 < \frac{x^{\theta}}{i} < q^\star \\
\paren{\kappa_1 -\kappa_0} A x^{\theta} - i & \frac{x^{\theta}}{i} \geq q^\star
\end{cases}
\end{equation}
for $(x,i) \in \R^+ \times \R^+$, where $q^\star$ is defined on \eqref{qstar} and $r_0$ is the positive root of the function $j$, defined on \eqref{j}.
\end{corollary}
 
\begin{proof}
Considering the relations of the functions $g$ and $f$ with the function $V$, it follows that the function $V$ proposed in Equation (\ref{V}) is the obvious candidate to be the solution of the HJB Equation \eqref{HJBoriginal}. So next we prove that indeed this is the case. In the following we use some results already presented in proof of Proposition \ref{Prop1Source}.
 
Let $(x,i) \in \R^+ \times \R^+$ and $q=\frac{x^\theta}{i}$.
Using some basic calculations one proves that $\mathcal{L}_Q V(x,i) = i \mathcal{L}_Q f(q)$ and $\rho V(x,i) - \mathcal{L} V(x,i) = i \prect{\rho f(q) - \mathcal{L}_Q f(q)}$.
Since $i$ is positive, the signal of $\rho V(x,i) - \mathcal{L} V(x,i)$ is exactly the signal of $\rho f(q) - \mathcal{L} _Qf(q)$. By construction, $\rho f(q) - \mathcal{L}_Q f(q) = 0$ in the continuation region and $\rho f(q) - \mathcal{L} f(q) \geq 0$ in the stopping region. Therefore, we just need to prove that in the continuation region $V(x,i)-g(x,i) \geq 0$. This follows immediately because $V(x,i)-g(x,i) = i \prect{f(q) - l(q)}$, $i$ is positive and we previously proved that $f(q) \geq l(q)$. 
 
Finally, from the initial and boundary conditions of the transformed problem \eqref{initial1uncert}, $V(0^+,i) = i f(0^+) = 0$ and for $(x^\star,i^\star) \in \partial C$, i.e., $\frac{{x^\star}^\theta}{i^\star}=q^\star$, we obtain the boundary conditions of the original problem \eqref{boundaryConditions}.
\end{proof}

\section{Comparative statics}
\label{statics}
In this section we study the behavior of the investment threshold as a function of the most relevant parameters, in particular the volatilities $\sigma_X$ and $\sigma_I$, the intensity rates of the jumps, $\lambda_X$ and $\lambda_I$, and the magnitude of the jumps, $m_X$ and $m_I$. In order to have analytical results, we assume that the magnitude of the jumps (both in the demand as in the investment) are deterministic and equal to $m_X$ and $m_I$, respectively\footnote{Numerical results suggest that the distribution of the jumps is not very relevant in terms of its effect on the investment thresholds.}.
 
Before we proceed with the mathematical derivations, we comment on the meaning of having $q^\star$ monotonic with respect to a specific parameter, as we are dealing with a boundary in a two dimensions space. For example, assuming that $q^\star$ increases with some parameter $y$, this would mean that $q^\star(y_1)<q^\star(y_2)$, for $y_1<y_2$ (leaving out the other parameters).
Consider then the corresponding continuation regions:
\begin{equation*}
\mathcal{C}_1 = \chav{(x,i) \in \R^+ \times \R^+ : 0 < \frac{x^{\theta}}{i} < q^\star_1}
\end{equation*}
and
\begin{equation*}
\mathcal{C}_2 = \chav{(x,i) \in \R^+ \times \R^+ : 0 < \frac{x^{\theta}}{i} < q^\star_2},
\end{equation*}
with $q_1^\star=q^\star (y_1)$ and $q_2^\star=q^\star (y_2)$.
In Figure \ref{fig:TwoTriggers}
we can observe the graphic representation of the boundaries which divide the stopping and the continuation regions, for each case, and, as expected, we would have $\mathcal{C}_1 \subset \mathcal{C}_2$. It follows from the splitting of the space into continuation and stopping, that you in this illustration, we would stop earlier for smaller values of the parameter.

\begin{center}
\begin{figure}[h!]
\centering
\begin{tikzpicture}
\draw[thick,->,ultra thick] (0,0) -- (5,0) node[anchor=north west] {$x$};
\draw[thick,->,ultra thick] (0,0) -- (0,4) node[left] {$i$};
\node[align=center] at (2.7,3.5) {$\partial \mathcal{C}_1$};
\node[align=center] at (4.6,3.5) {$\partial \mathcal{C}_2$};
\draw[domain=0:2.4,smooth,variable=\x,ultra thick] plot ({\x},{(1/1.5)*\x*\x});
\draw[domain=0:4.2,smooth,variable=\x,ultra thick,dashed] plot ({\x},{(1/5)*\x*\x});
\end{tikzpicture}
\caption{The full (respectively, dashed) line is the boundary of the $\mathcal{C}_1$ (respectively, $\mathcal{C}_2$) set.}
\label{fig:TwoTriggers}
\end{figure}
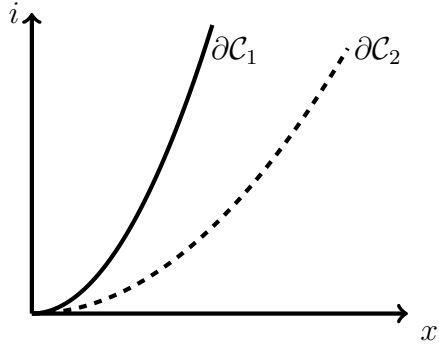
\end{center}

In the previous sections we consider $j$ only as a function of $r$. As in this section we study the behavior of $j$ with respect to the parameters of both processes involved, we change slightly the notation, in order to emphasize the dependency of $j$ in such parameters. Therefore, we define the vectors
$\boldsymbol{\phi_X}=\paren{\mu_X,\sigma_X, \lambda_X, m_X}$, $\boldsymbol{\phi_I}= \paren{\mu_I,\sigma_I, \lambda_I, m_I}$ ,
$\boldsymbol{\varsigma} = \paren{\theta, \rho, \boldsymbol{\phi_X}}$ and finally $\boldsymbol{\varrho} = \paren{\theta, \rho, \boldsymbol{\phi_X}, \boldsymbol{\phi_I}}$, that we use in the following definition:
 
\begin{eqnarray*}
j(r,\boldsymbol{\varrho}) &=& \prect{\frac{\sigma_X^2}{2}\theta^2 + \frac{\sigma_I^2}{2}} r^2 + \prect{\paren{\mu_X -\frac{\sigma_X^2}{2}-\lambda_X m_X} \theta-\paren{\mu_I - \frac{\sigma_I^2}{2} -\lambda_I m_I}} r \\
&& + \paren{\mu_I-\lambda_I m_I}-\paren{\lambda_X + \lambda_I }-\rho +\lambda_X (1+m_X)^{\theta r}+\lambda_I (1+m_I)^{1-r}.
\end{eqnarray*}
Furthermore, we let
$r_0 (\boldsymbol{\varrho} )$ denote the positive root of $j$ and we change accordingly the definition of $h$ as follows:
$$
\h (\boldsymbol{\varsigma}) = \rho - \paren{\mu_X+(\theta-1) \frac{\sigma_X^2}{2}-\lambda_X m_X} \theta - \lambda_X \paren{\paren{1+m_X}^{\theta}-1}
$$
and finally
$$ q^\star (\boldsymbol{\varrho}) = \frac{1}{\paren{\kappa_1-\kappa_0} A (\boldsymbol{\varsigma})} \;\frac{r_0(\boldsymbol{\varrho})}{r_0(\boldsymbol{\varrho})-1},$$
with $A (\boldsymbol{\varsigma})$ defined on \eqref{A}.

We start by proving the following lemma:
\begin{lem}
\label{lemam}
The investment threshold changes with the parameters according to the following relation:
\begin{equation*}
\frac{d q^\star (\boldsymbol{\varrho})}{d y} = \frac{1}{\Delta (\kappa_1-\kappa_0) A (\boldsymbol{\varsigma}) (r_0(\boldsymbol{\varrho})-1)^2} \times \begin{cases}
\frac{\Delta r_0 (\boldsymbol{\varrho})(r_0 (\boldsymbol{\varrho}) -1)(n \h (\boldsymbol{\varsigma}) +1)}{\h (\boldsymbol{\varsigma})}\deriv{\h (\boldsymbol{\varsigma})}{y}+ \barra{\deriv{j(r,\boldsymbol{\varrho})}{y}}_{r=r_0} & \text{if} \; y \in \boldsymbol{\phi_X} \\ \\
\barra{\deriv{j(r,\boldsymbol{\varrho})}{y}}_{r=r_0} & \text{if} \; y \in \boldsymbol{\phi_I}
\end{cases}
\end{equation*}
with $\Delta= \barra{\deriv{j(r,\boldsymbol{\varrho})}{r}}_{r=r_0}$. Furthermore, as $\Delta>0$, then the sign of the derivative of $q^{\star}$ depends only on the derivatives of $h$ and $j$ with respect to the study parameter.
\end{lem}
\begin{proof}
As
\begin{equation*}
\frac{d q^\star (\boldsymbol{\varrho})}{d y} = \frac{1}{(\kappa_1-\kappa_0) A (\boldsymbol{\varsigma}) (r_0(\boldsymbol{\varrho})-1)^2} \times \begin{cases}
\frac{r_0 (\boldsymbol{\varrho})(r_0 (\boldsymbol{\varrho})-1)(n \h (\boldsymbol{\varsigma})+1)}{\h (\boldsymbol{\varsigma})}\deriv{\h (\boldsymbol{\varsigma})}{y}-\deriv{r_0(\boldsymbol{\varrho})}{y} & \text{if} \; y \in \boldsymbol{\phi_X} \\
-\deriv{r_0(\boldsymbol{\varrho})}{y} & \text{if} \; y \in \boldsymbol{\phi_I}
\end{cases}
\end{equation*}
the result  follows from the implicit derivative relation $\frac{d r_0 (\boldsymbol{\varrho})}{d y} = \barra{- \frac{\deriv{j(r,\boldsymbol{\varrho})}{y}}{\deriv{j(r,\boldsymbol{\varrho})}{r}}}_{r=r_0}$.
 
It remains to prove that $\Delta>0$. As a function of $r$ only, $j$ is continuous and strictly convex, with $lim_{r\rightarrow +\infty} j(r) = + \infty$, $j(0)<0$ and $j(1)<0$, as we have already stated. Since $r_0$ is its positive root, in view of the properties of $j$, it follows that it must be increasing in a neighborhood of $r_0$, and thus $\Delta > 0$.
\end{proof}

In the following sections, we study the sign of each derivative, whenever possible.
Our goal is to inspect if the trigger value, $q^\star$, increases or decreases with each parameter in $\boldsymbol{\phi_X}$ and $\boldsymbol{\phi_I}$.

\subsection{Comparative statics for the parameters regarding the demand}
In this section we state and prove the results concerning the behavior of $q^\star$ as a function of the demand parameters.
Before the main proposition of this subsection, we present and prove a Lemma that provide us the admissible values for the parameters.
 
 
\begin{lem}
\label{lema:restrictions}
In view of the Assumption \ref{ass}, the following restrictions on the demand parameters hold.
The drift is upper bounded, i.e.,$$\mu_X < \frac{1}{\theta}\prect{\rho - \paren{(\theta-1) \frac{\sigma_X^2}{2}-\lambda_X m_X} \theta - \lambda_X \paren{\paren{1+m_X}^{\theta}-1} - \sigma^2_I}.$$
The domain of the volatility parameter depends on $\theta$. If $0<\theta<1$, $\sigma_X$ is lower bounded, $\sigma_X > B$, otherwise is  upper bounded, $0 < \sigma_X < B$, where
$$B = \sqrt{ \max \chav{0,\frac{2}{\theta (\theta -1)} \prect{ \rho - \paren{\mu_X - \lambda_X m_X} \theta - \lambda_X \paren{\paren{1+m_X}^{\theta}-1} -\sigma^2_I}}}.$$
Relatively to the rate of the jumps, the domain also depends on $\theta$. If $0<\theta<1$, $\lambda_X$ is lower bounded, $\lambda_X > C$, otherwise is  upper bounded, $0 < \lambda_X < C$, where
$$C = \max \chav{0, \frac{\rho - \paren{\mu_X + (\theta-1) \frac{\sigma^2_X}{2}}\theta - \sigma^2_I}{(1+m_X)^\theta - (\theta m_X+1)}}. $$
%
%
%
%
Regarding the jump sizes, we have the following: let $\xi_1$ and $\xi_2$ be the negative and positive solutions of the equation
\begin{equation}
\label{zetas}
(1+x)^\theta-(\theta x+1) = \frac{\rho - \paren{\mu_X+(\theta-1)\frac{\sigma^2_X}{2}} \theta - \sigma_I^2}{\lambda_X},
\end{equation}
respectively, when there are solutions. Moreover, let
\begin{equation}
\label{D}
D=\frac{\rho - \paren{\mu_X+(\theta-1)\frac{\sigma^2_X}{2}} \theta - \sigma_I^2}{\lambda_X} .
\end{equation}
For $0 < \theta<1 $,
\begin{itemize}
\item if $D \leq \theta-1$ then $m_X >\xi_2$;
\item if $\theta-1 < D < 0 $ then $m_X \in (-1,\xi_1) \cup (\xi_2,+\infty)$;
\item if $D\geq 0$ then $m_X>-1$.
\end{itemize} For $\theta > 1 $,
\begin{itemize}
\item if $0 < D <\theta-1$ then $m_X \in (\xi_1,\xi_2)$;
\item if $D \geq \theta-1$ then $m_X \in (-1,\xi_2)$ \footnote{The case $D\leq0$ is equivalent to $h\leq \sigma_I^2$, and therefore Assumption \ref{ass} does not hold. },
\end{itemize}
and if $\theta=1$ and $D>0$, then $m_X>-1$.
\end{lem}

\begin{proof}
The sets of admissible values for the parameters $\mu_X$ and $\sigma_X$ come from straightforward manipulation of the second restriction of Assumption \ref{ass}.
 
The results regarding the jumps parameters ($\lambda_X$ and $m_X$) rely on the properties of the function
\begin{equation}
\label{Up}
\Upsilon_\alpha(x)=\Upsilon_{\alpha,2}(x)-\Upsilon_{\alpha,1}(x),
\end{equation}
with $\Upsilon_{\alpha,1}(x)= \alpha x + 1$ and $\Upsilon_{\alpha,2}(x)= (1+x)^\alpha,$
with $x>-1$ and $\alpha>0$. The graph of the function $\Upsilon_{\alpha,1}$ is a tangent line of the graph of the function $\Upsilon_{\alpha,2}$ on the point $(0,1)$. Also, as $\Upsilon^{\prime \prime}_{\alpha,2}(x)= \alpha (\alpha -1) (1+x)^{\alpha-2}$, the function $\Upsilon_{\alpha,2}$ is convex if $\alpha > 1$ and it is concave otherwise. Then, for $x>-1$, $\Upsilon_{\alpha,2}(x)>\Upsilon_{\alpha,1}(x)$ if $\alpha >1$ and $\Upsilon_{\alpha,1}(x) \geq \Upsilon_{\alpha,2}(x)$ if $0<\alpha \leq 1$.
Then it follows that $\Upsilon_\alpha(x)> 0$ if $\alpha>1$ and $\Upsilon_\alpha(x) \leq 0$ if $0<\alpha \leq 1$.
 
With respect to $\lambda_X$, in view of the restriction $h>\sigma_I^2$, $$\lambda_X \Upsilon_\theta(m_X)<\rho-\paren{\mu_X+(\theta-1)\frac{\sigma_X^2}{2}}\theta-\sigma^2_I$$ and therefore the restriction on this parameter follows.
 
Proceeding with the analysis with respect to the jump size, $m_X$, we start by noting that $\Upsilon_\alpha(-1)=\alpha-1$, $\Upsilon_\alpha(0)=0$ and $\lim_{x \rightarrow \infty} \Upsilon_\alpha(x)=+\infty \Indicatriz{\alpha>1}-\infty \Indicatriz{\alpha<1}$. From these properties we know the shape of the function $\Upsilon_\alpha$ for each $\alpha$. The restriction $h>\sigma_I^2$ is equivalent to have $\Upsilon_\theta(m_X)<D$ and therefore standard study of functions lead us to the results presented in the Lemma.
 
\end{proof}
 
Now we are in position to state the main properties of $q^\star$ as a function of the demand parameters. From Lemma \ref{lemam}, we can only get results about the behavior of $q^\star$ as a function of the parameters if $\deriv{\h (\boldsymbol{\varsigma}) }{y}$ and $\barra{\deriv{j(r,\boldsymbol{\varrho})}{y}}_{r=r_0}$ have the same sign. Therefore we are not able to present a comprehensive study of $q^\star$ for all values of $\theta$.
 
\begin{proposition}
For $\frac{1}{r_0(\boldsymbol{\varrho})} \leq \theta \leq 1$, the investment threshold $q^{\star}$:
\begin{itemize}
\item increases with the demand volatility, $\sigma_X$ and jump intensity, $\lambda_X$;
\item is not monotonic with respect to the magnitude of the jumps in the demand, $m_X$.
In fact,
\begin{itemize}
\item if $D \leq \theta-1$ , then $q^\star$ increases for $m_X >\xi_2$;
\item if $\theta-1 < D \leq 0 $, then $q^\star $ decreases for $m_X \in (-1,\xi_1)$ and increases for $ m_X >\xi_2$;
\item if $D>0$, then $q^\star $ decreases for $m_X \in (-1,0) $ and increases for $m_X>0$,
\end{itemize}
where $D$ is defined on \eqref{D} and $\xi_1$ and $\xi_2$ are,respectively, the negative and positive solutions of the Equation \eqref{zetas}, when there are solutions.
\end{itemize}
\end{proposition}
 
\begin{proof}
 
The proof for $\sigma_X$ is straightforward knowing the derivatives
\begin{equation*}
\deriv{\h (\boldsymbol{\varsigma})}{\sigma_X}=-\theta (\theta - 1) \sigma_X \; \; \text{and} \; \; \barra{\deriv{j(r,\boldsymbol{\varrho})}{\sigma_X}}_{r=r_0} = \theta r_0 (\boldsymbol{\varrho}) (\theta r_0 (\boldsymbol{\varrho}) - 1) \sigma_X .
\end{equation*}
 
Continuing with $\lambda_X$, we calculate the derivatives
\begin{align*}
\deriv{\h (\boldsymbol{\varsigma})}{\lambda_X}&= \theta m_X + 1 - (1+m_X)^\theta=-\Upsilon_\theta(m_X) \\ \barra{\deriv{j(r,\boldsymbol{\varrho})}{\lambda_X}}_{r=r_0} &= (1+m_X)^{\theta r_0 (\boldsymbol{\varrho})}- \paren{\theta r_0 (\boldsymbol{\varrho}) m_X + 1}=\Upsilon_{\theta r_0 (\boldsymbol{\varrho})}(m_X),
\end{align*}
with $\Upsilon_\alpha$ given by \eqref{Up}. Taking into account the properties of such function, we conclude the result for $\lambda_X$.
 
Finally, for $m_X$ we also calculate the derivatives
\begin{align*}
\deriv{\h (\boldsymbol{\varsigma})}{m_X}&= \theta \lambda_X \prect{1 -(1+m_X)^{\theta-1}}=\theta \lambda_X \Phi_\theta(m_X)\\
\barra{\deriv{j(r,\boldsymbol{\varrho})}{m_X}}_{r=r_0} &= - \theta r_0 (\boldsymbol{\varrho}) \lambda_X \prect{1 -(1+m_X)^{\theta r_0 (\boldsymbol{\varrho})-1}}=-\theta r_0 (\boldsymbol{\varrho}) \lambda_X \Phi_{\theta r_0 (\boldsymbol{\varrho})}(m_X)
\end{align*}
where
\begin{equation}
\label{phifunction}
\Phi_\alpha (x) = 1-(1+x)^\alpha,
\end{equation}
with $x>-1$ and $\alpha>0$. As $\Phi_\alpha^\prime(x)=-\alpha (1+x)^{\alpha -1}$, $\Phi_\alpha$ is decreasing if $\alpha>0$, constant and equal to zero if $\alpha=0$, and increasing if $\alpha < 0$. Moreover, as $\Phi_\alpha (0)=0$, then:
\begin{itemize}
\item if $\alpha>0$, $\Phi_\alpha$ is positive when $-1<x<0$ and is negative for $x>0$;
\item if $\alpha<0$, $\Phi_\alpha$ is negative when $-1<x<0$ and positive for $x>0$.
\end{itemize}
Given these results, the proof regarding $m_X$ is concluded.
\end{proof}

\bigskip
\begin{rem}
Numerical examples show us that the investment threshold $q^\star$ is not monotonic with respect to the demand's drift, $\mu_X$. We cannot have analytical results for the behavior of $q^\star$ with respect to $\mu_X$ as the signs of the involved derivatives are opposite: $\deriv{\h (\boldsymbol{\varsigma})}{\mu_X}=-\theta$ and $\barra{\deriv{j(r,\boldsymbol{\varrho})}{\mu_X}}_{r=r_0}=\theta r_0 (\boldsymbol{\varrho})$.
\end{rem}
 
\subsection{Comparative statics for the parameters regarding the investment}
Now we state and prove the results concerning the behavior of $q^\star$ as a function of the investment parameters. In view of Lemma \ref{lemam}, we just need to assess the sign of the derivative of $j$ with respect to each one of the parameters.
\begin{proposition} For values of investment parameters such that Assumption \ref{ass} holds, the investment threshold $q^{\star}$
\begin{itemize}
\item decreases with the investment drift $\mu_I$;
\item increases with the investment volatility, $\sigma_I$ and jumps intensity, $\lambda_I$;
\item is not monotonic with respect to the magnitude of the jumps in the intensity; it
decreases when $m_I \in (-1,0)$ and increases when $m_I>0$.
\end{itemize}
\end{proposition}
\begin{proof}
The proof for $\mu_I$ and $\sigma_I$ is trivial.
In order to pove the result for $\lambda_I$, we study the behavior of the function $\Theta$:
\begin{eqnarray*}
\Theta(r) = \deriv{j(r,\boldsymbol{\varrho})}{\lambda_I}= m_I (r-1) + (1+m_I)^{1-r} - 1.
\end{eqnarray*}
We compute the second derivative,
\begin{equation*}
\Theta^{\prime \prime} (r) = \prect{\ln (1+m_I)}^2 (1+m_I)^{1-r} >0,
\end{equation*}
whence $\Theta$ is a strictly convex function. Furthermore, $\Theta(0)=\Theta(1)=0$, which lead us to conclude that $\Theta(r)>0$, for $r>1$. Therefore, $\barra{\deriv{j(r,\boldsymbol{\varrho})}{\lambda_I}}_{r=r_0}=\Theta(r_0 (\boldsymbol{\varrho}) )>0$, and the result follows.
 
Finally, in order to study the behavior of $q^\star$ with respect to $m_I$, we need to assess the sign of the following derivative:
\begin{eqnarray*}
\barra{\deriv{j(r,\boldsymbol{\varrho})}{m_I}}_{r=r_0}&=& \lambda_I (r_0 (\boldsymbol{\varrho})-1)\prect{1-(1+m_I)^{-r_0 (\boldsymbol{\varrho})}}.
\end{eqnarray*}
Taking into account the $\Phi_\alpha$ function defined on \eqref{phifunction}, we notice that $\barra{\deriv{j(r,\boldsymbol{\varrho})}{m_I}}_{r=r_0}=\lambda_I (r_0 (\boldsymbol{\varrho})-1) \Phi_{-r_0 (\boldsymbol{\varrho})}(m_X)$.
As $r_0 (\boldsymbol{\varrho})>1$ and in view of the properties of function $\Phi_\alpha$, the result follows.
\end{proof}

\section{Conclusion}
In this paper we proposed a new way to derive analytically the solution to a decision problem with two uncertainties, both following jump-diffusion processes. This method relies on a change of variable that reduces the problem from a two dimensions problem to a one dimension, which we can solve analytically.
 
Moreover, we also presented an extensive comparative statics. From such analysis, we have concluded that the behavior of the investment threshold with respect to demand and investment parameters is similar, i.e., the impact of changing one parameter in the demand process is qualitatively the same as changing the same parameter in the investment process. This result holds for all except for the drift: the investment threshold is monotonic with respect to the investment drift but this does not hold for the demand drift.
 
For some parameters, the result is as expected, and similar to the one derived in earlier works (like the behavior of the investment threshold with the volatility of both the demand and the investment), whereas in other cases the results are not so straightforward. For instance, the impact of the jump sizes of the demand depend on some analytical condition whose economical interpretation is far from being obvious.
 
Our method may be extended for more general profit functions. In our research agenda we plan to study properties of the functions such that the approach proposed in this paper will still be valid and useful.

\section{Proofs}
\label{sec::proofs}
 
\subsection{Proof of Lemma \ref{propEY}}
\label{moments}
\begin{proof}
Applying expectation and re-writting $Y_{t+s}$, we have
\begin{equation*}
\E \prect{\barra{ Y_{t+s}^k} {Y_t}} = Y_t^k \exp \prect{\paren{\mu-\frac{\sigma^2}{2}-\lambda m} k s} \; \E \prect{\barra{\exp \prect{\sigma k \paren{W_{t+s} - W_t}} \prod_{i=N_t+1}^{N_{t+s}}\paren{1+U_i}^k} {Y_t}}.
\end{equation*}
Using the fact that the Brownian motion and the Poisson process have independent and stationary increments, and that the processes are independent, we obtain
\begin{equation}\label{EV}
\E \prect{\barra{ Y_{t+s}^k} {Y_t}} = Y_t^k \exp \prect{\paren{\mu-\frac{\sigma^2}{2}-\lambda m} k s} \E \prect{\exp \paren{\sigma k W_s}} \E \prect{\prod_{i=1}^{N_s}\paren{1+U_i}^k}.
\end{equation}
As $\chav{U_i}_{i \in \N} \overset{i.i.d.}{\thicksim} U$, then using the probability generating function of a Poisson distribution and the tower property of the conditional expectation, we conclude that
\begin{equation}\label{pgfP}
\E \prect{\prod_{i=1}^{N_s}\paren{1+U_i}^k} = \exp \prect{\lambda s \paren{\E \prect{\paren{1+U}^k}-1}}.
\end{equation}
The result follows using the moment generating function of the Normal distribution, and replacing \eqref{pgfP} in \eqref{EV}.
\end{proof}


\begin{thebibliography}{10}

\bibitem{Dixit:Pindyck:94}
A.~Dixit and R.~Pindyck.
\newblock {\em Investment Under Uncertainty}.
\newblock Princeton University Press, 1994.

\bibitem{Eberlein:10}
E.~Eberlein.
\newblock Jump processes.
\newblock {\em Encyclopedia of Quantitative Finance, R. Cont (ed.), John Wiley
  \& Sons Ltd.}, pages 990--994, 2014.

\bibitem{Hagspiel:14}
V.~Hagspiel, K.~J.M. Huisman, and C.~Nunes.
\newblock Optimal technology adoption when the arrival rate of new technologies
  changes.
\newblock {\em European Journal of Operational Research}, 243:897--911, 2015.

\bibitem{Koh:07}
W.~Koh and D.~Paxson.
\newblock Real extreme {R\&D} discovery options.
\newblock {\em Banque and Marches}, 86, 2007.

\bibitem{Martzoukos:02}
S.~H. Martzoukos and L.~Trigeorgis.
\newblock Real (investment) options with multiple sources of rare events.
\newblock {\em European Journal of Operational Research}, 136:696--706, 2002.

\bibitem{Oksendal:03}
B.~{\O}ksendal and A.~Sulem.
\newblock {\em Applied Stochastic Control of Jump Diffusions}.
\newblock Springer, 2007.

\bibitem{Wu:12}
L.C. Wu, S.~H. Li, C.~S. Ong, and C.~Pan.
\newblock Options in technology investment games: The real world {TFT-LCD}
  industry case.
\newblock {\em Technological Forecasting \& Social Change}, 79:1241--1253,
  2012.

\bibitem{Wu:07}
M.~C. Wu and S.~H. Yen.
\newblock Pricing real growth options when the underlying assets have jump
  diffusion processes: the case of {R\&D} investments.
\newblock {\em R\&D Management}, 37, 2007.

\bibitem{Yang:07}
M.~Yang, W.~Blyth, M.~Yang, and W.~Blyth.
\newblock Modeling investment risks and uncertainties with real options
  approach, 2007.

\bibitem{yong:99}
J.~Yong and X.~Y. Zhou.
\newblock {\em Stochastic controls: Hamiltonian systems and HJB equations},
  volume~43.
\newblock Springer Science \& Business Media, 1999.

\end{thebibliography}
\end{document}